\documentclass[11pt,reqno]{amsart}

\usepackage{xspace}
\usepackage{amsmath,amsthm,amsfonts,amssymb,latexsym,mathrsfs,color,extarrows}
\usepackage{hyperref}

\newtheorem{theorem}{Theorem}
\newtheorem{cor}[theorem]{Corollary}
\newtheorem{prop}[theorem]{Proposition}

\newcommand{\single}{{\rm single\,}}

\newcommand{\des}{{\rm des\,}}

\newcommand{\exc}{{\rm exc\,}}
\newcommand{\aexc}{{\rm aexc\,}}
\newcommand{\we}{{\rm wexc\,}}

\newcommand{\cyc}{{\rm cyc\,}}
\newcommand{\fix}{{\rm fix\,}}
\newcommand{\dc}{{\rm dc\,}}

\newcommand{\mdn}{\mathcal{D}}

\newcommand{\msn}{\mathfrak{S}_n}

\newcommand{\ms}{\mathfrak{S}}

\newcommand{\lrf}[1]{\lfloor #1\rfloor}

\newcommand{\asc}{{\rm asc\,}}

\newcommand{\Eulerian}[2]{\genfrac{<}{>}{0pt}{}{#1}{#2}}

\newcommand{\arxiv}[1]{\href{http://arxiv.org/abs/#1}{\texttt{arXiv:#1}}}
\title[Several polynomials associated with Eulerian polynomials]{Context-free grammars for several polynomials associated with Eulerian polynomials}
\author[S.-M.~Ma]{Shi-Mei Ma}
\address{School of Mathematics and Statistics,
        Northeastern University at Qinhuangdao,
         Hebei 066004, P.R. China}
\email{shimeimapapers@163.com (S.-M. Ma)}
\author[J.~Ma]{Jun Ma}
\address{Department of mathematics, shanghai jiao tong university, shanghai, china}
\email{majun904@sjtu.edu.cn(J.~Ma)}
\author[Y.-N. Yeh]{Yeong-Nan Yeh}
\address{Institute of Mathematics,
        Academia Sinica, Taipei, Taiwan}
\email{mayeh@math.sinica.edu.tw (Y.-N. Yeh)}
\author[B.-X. Zhu]{Bao-Xuan Zhu}
\address{School of Mathematical Sciences, Jiangsu Normal University, Xuzhou 221116, P.R. China}
\email{bxzhu@jsnu.edu.cn (B.-X. Zhu)}
\begin{document}

\maketitle
\begin{abstract}
In this paper, we present grammatical descriptions of several polynomials associated with Eulerian polynomials,
including $q$-Eulerian polynomials, alternating run polynomials and derangement polynomials. As applications, we get
several convolution formulas involving these polynomials. 
\bigskip

\end{abstract}
\date{\today}
\section{Introduction}
For an alphabet $A$, let $\mathbb{Q}[[A]]$ be the rational commutative ring of formal power
series in monomials formed from letters in $A$. Following~\cite{Chen93}, a context-free grammar over
$A$ is a function $G: A\rightarrow \mathbb{Q}[[A]]$ that replace a letter in $A$ by a formal function over $A$.
The formal derivative $D$ is a linear operator defined with respect to a context-free grammar $G$. More precisely,
the derivative $D=D_G$: $\mathbb{Q}[[A]]\rightarrow \mathbb{Q}[[A]]$ is defined as follows:
for $x\in A$, we have $D(x)=G(x)$; for a monomial $u$ in $\mathbb{Q}[[A]]$, $D(u)$ is defined so that $D$ is a derivation,
and for a general element $q\in\mathbb{Q}[[A]]$, $D(q)$ is defined by linearity. 

Many combinatorial structures can be generated by using context-free
grammars, such as set partitions~\cite{Chen93}, permutations~\cite{Dumont96,Ma121,Ma132},
Stirling permutations~\cite{Chen16,Ma16}, increasing trees~\cite{Chen16,Dumont96},
rooted trees~\cite{Dumont9602} and perfect matchings~\cite{Ma131}.
In this paper, we shall use
a grammatical labeling introduced by Chen and Fu~\cite{Chen16} to
study several polynomials associated with Eulerian polynomials.

Let $\msn$ be the symmetric group on the set $[n]=\{1,2,\ldots,n\}$.
Let $\pi=\pi(1)\pi(2)\cdots\pi(n)\in\msn$. We say that
the index $i\in [n-1]$ is an {\it excedance} of $\pi$ if $\pi(i)>i$.
Denote by $\exc(\pi)$ the number of excedances of $\pi$. The {\it Eulerian polynomials} are defined by
\begin{equation*}
A_0(x)=1,~A_n(x)=\sum_{\pi\in\msn}x^{\exc(\pi)}=\sum_{k=0}^{n-1}\Eulerian{n}{k}x^{k}~\textrm{for $n\ge 1$},
\end{equation*}
where $\Eulerian{n}{k}$ are the {\it Eulerian numbers}.
The numbers $\Eulerian{n}{k}$
satisfy the recurrence relation
$$\Eulerian{n}{k}=(k+1)\Eulerian{n-1}{k}+(n-k)\Eulerian{n-1}{k-1},$$
with the initial conditions $\Eulerian{0}{0}=1$ and $\Eulerian{0}{k}=0$ for $k\geq 1$.

The {\it hyperoctahedral group} $B_n$ is the group of signed permutations of the set $\pm[n]$ such that $\pi(-i)=-\pi(i)$ for all $i$, where $\pm[n]=\{\pm1,\pm2,\ldots,\pm n\}$. Throughout this paper, we always identify a signed permutation $\pi=\pi(1)\cdots\pi(n)$ with the word $\pi(0)\pi(1)\cdots\pi(n)$, where $\pi(0)=0$.
Let $\des_B(\pi)=\#\{i\in\{0,1,2,\ldots,n-1\}|\pi(i)>\pi({i+1})\}$. The {\it Eulerian polynomials of type $B$} are defined by
$${B}_n(x)=\sum_{\pi\in B_n}x^{\des_B(\pi)}=\sum_{k=0}^nB(n,k)x^{k},$$
where $B(n,k)$ are called the {\it Eulerian numbers of type $B$}.
The numbers $B(n,k)$ satisfy
the recurrence relation
$$B(n+1,k)=(2k+1)B(n,k)+(2n-2k+3)B(n,k-1)$$
for $n,k\geq 0$, where $B(0,0)=1$ and $B(0,k)=0$ for $k\geq 1$.
Let us now recall two results on
context-free grammars.
\begin{prop}[{\cite[Section~2.1]{Dumont96}}]\label{Dumont96}
If $A=\{x,y\}$ and $G=\{x\rightarrow xy, y\rightarrow xy\}$,
then for $n\ge 1$
\begin{equation*}
D^n(x)=x\sum_{k=0}^{n-1}\Eulerian{n}{k}x^{k}y^{n-k}.
\end{equation*}
\end{prop}

\begin{prop}[{\cite[Theorem~10]{Ma131}}]\label{Ma13}
If $A=\{x,y\}$ and $G=\{x\rightarrow xy^2, y\rightarrow x^2y\}$,
then for $n\geq 1$,
$$D^n(x^2)=2^n\sum_{k=0}^{n-1}\Eulerian{n}{k}x^{2n-2k}y^{2k+2},$$
$$D^n(xy)=\sum_{k=0}^nB(n,k)x^{2n-2k+1}y^{2k+1}.$$
\end{prop}

The {\it $q$-Eulerian polynomials} of types $A$ and $B$ are respectively defined by
$$A_n(x;q)=\sum_{\pi\in\msn}x^{\exc(\pi)}q^{\cyc(\pi)},$$
$$B_n(x;q)=\sum_{\pi\in B_n}x^{\des_B(\pi)}q^{N(\pi)},$$
where $\cyc(\pi)$ is the number of cycles in $\pi$ and $N(\pi)=\#\{i\in [n]: \pi(i)<0\}$.
Following~\cite[157-162]{DB62} and~\cite[Theorem~4.3.3]{Zhao11},
the {\it alternating run polynomials} of types $A$ and $B$ can be respectively defined by
\begin{align*}
R_n(x)&=(1-w)\left(\dfrac{1+x}{2}\right)^{n-1}(1+w)^{n}A_n\left(\dfrac{1-w}{1+w}\right),\\
T_n(x)&=\frac{x}{2}\left(\dfrac{1+x}{2}\right)^{n-1}(1+w)^{n}B_n\left(\dfrac{1-w}{1+w}\right)
\end{align*}
for $n\geq 2$,
where $w=\sqrt{\frac{1-x}{1+x}}$.
Following~\cite[Proposition~5]{Brenti90} and~\cite[Theorem~3.2]{Chow09},
the {\it derangement polynomials} of types $A$ and $B$ can be respectively defined by
$$d_n(x)=\sum_{k=0}^n(-1)^{n-k}\binom{n}{k}A_k(x),$$
$$d_n^B(x)=\sum_{k=0}^n(-1)^{n-k}x^{n-k}\binom{n}{k}B_k(x).$$

The purpose of this paper is to explore grammatical descriptions of
the $q$-Eulerian polynomials, the alternating run polynomials and the derangement polynomials.
\section{$q$-Eulerian polynomials}\label{Section:02}
According to~\cite[Proposition 7.2]{Bre00}, the polynomials $A_n(x;q)$ satisfy the recurrence relation
\begin{equation}\label{anxq-rr}
A_{n+1}(x;q)=(nx+q)A_{n}(x;q)+x(1-x)\frac{\partial}{\partial x}A_{n}(x;q),
\end{equation}
with the initial condition $A_0(x;q)=1$.
In~\cite[Theorem~3.4]{Bre94}, Brenti showed that the polynomials $B_n(x;q)$ satisfy the recurrence relation
\begin{equation}\label{bnxq-rr}
B_{n+1}(x;q)=((n+nq+q)x+1)B_{n}(x;q)+(1+q)x(1-x)\frac{\partial}{\partial x}B_{n}(x;q),
\end{equation}
with the initial condition $B_0(x;q)=1$.
The polynomials $A_n(x;q)$ and $B_n(x;q)$ have been extensively studied (see~\cite{Bra06,Ma08} for instance).
\subsection{Context-free grammar for $A_n(x;q)$}
\hspace*{\parindent}

A {\it standard cycle
decomposition} of $\pi\in\mathfrak{S}_n$ is defined by requiring that each cycle is
written with its smallest element first, and the cycles are
written in increasing order of their smallest element. In the following discussion, we write $\pi$ in standard cycle decomposition.
We say that the index $i\in [n]$ is an {\it anti-excedance} of $\pi$ if $\pi(i)\leq i$.
Denote by $\aexc(\pi)$ the number of anti-excedances of $\pi$.
It is clear that $\exc(\pi)+\aexc(\pi)=n$.
We now give a labeling of $\pi$ as follows:
\begin{itemize}
  \item [\rm ($i$)]Put a superscript label $z$ right after each excedance;
  \item [\rm ($ii$)]Put a superscript label $y$ right after each anti-excedance;
  \item [\rm ($iii$)]Put a superscript label $q$ before each cycle of $\pi$;
  \item [\rm ($iv$)]Put a superscript label $x$ at the end of $\pi$
\end{itemize}
For example,
the permutation $(1,3,4)(2)(5,6)$ can be labeled as
$^q(1^z3^z4^y)^q(2^y)^q(5^z6^y)^x$. The weight of $\pi$ is the product of its labels.

Let $\msn(i,j,k)=\{\pi\in\msn: \aexc(\pi)=i, \exc(\pi)=j, \cyc(\pi)=k\}$.
When $n=1$, we have $\ms_1(1,0,1)=\{^q(1^y)^x\}$. When $n=2$, we have
$\ms_2(2,0,2)=\{^q(1^y)^q(2^y)^x\}$ and $\ms_2(1,1,1)=\{^q(1^z2^y)^x\}$.
Let $n=m$. Suppose we get all labeled permutations in $\msn(i,j,k)$ for all $i,j,k$.
Let $\pi'\in\ms_{n+1}$ be obtained from $\pi\in\msn(i,j,k)$ by inserting the entry $n+1$ into $\pi$.
We distinguish the following three cases:
\begin{itemize}
  \item [\rm ($c_1$)] If the entry $n+1$ is inserted as a new cycle $(n+1)$, then $\pi'\in\ms_{n+1}(i+1,j,k+1)$.
  In this case, the insertion of $n+1$ corresponds to the operation $x\rightarrow qxy$;
\item [\rm ($c_2$)]If the entry $n+1$ is inserted right after an excedance, then $\pi'\in\ms_{n+1}(i+1,j,k)$.
In this case, the insertion of $n+1$ corresponds to the operation $z\rightarrow yz$;
\item [\rm ($c_3$)]If the entry $n+1$ is inserted right after an anti-excedance, then $\pi'\in\ms_{n+1}(i,j+1,k)$.
In this case, the insertion of $n+1$ corresponds to the operation $y\rightarrow yz$.
\end{itemize}
By induction, we get the following result.
\begin{theorem}\label{thm:01}
If $A=\{x,y,z\}$ and $G=\{x\rightarrow qxy, y\rightarrow yz, z\rightarrow yz\}$,
then
\begin{equation}\label{Dnx-Euleiran}
D^n(x)=x\sum_{\pi\in\msn}y^{\aexc(\pi)}z^{\exc(\pi)}q^{\cyc(\pi)}.
\end{equation}
\end{theorem}
Setting $y=1$ in~\eqref{Dnx-Euleiran}, we get $$D^n(x)|_{y=1}=xA_n(z;q).$$
Furthermore, setting $y=z=1$ in~\eqref{Dnx-Euleiran}, we get
$$D^n(x)|_{y=z=1}=x\sum_{\pi\in\msn}q^{\cyc(\pi)}=xq(q+1)(q+2)\cdots (q+n-1).$$

\begin{cor}
For $n\geq 1$, we have
$$A_{n+1}(x;q)=qA_n(x;q)+qx\sum_{k=0}^{n-1}\binom{n}{k}A_k(x;q)A_{n-k}(x).$$
\end{cor}
\begin{proof}
Using the {\it Leibniz's formula},
we get
$$D^{n+1}(x)=D^n(qxy)=qD^n(xy)=q\sum_{k=0}^n\binom{n}{k}D^k(x)D^{n-k}(y).$$
Combining Proposition~\ref{Dumont96} and Theorem~\ref{thm:01}, we obtain
$$A_{n+1}(x;q)=q\sum_{k=0}^n\binom{n}{k}A_k(x;q)x^{n-k}A_{n-k}\left(\frac{1}{x}\right).$$
Recall that $A_n(x)$ are symmetric, i.e., $x^nA_n\left(\frac{1}{x}\right)=xA_n(x)$ for $n\ge 1$. Thus we get the desired result.
\end{proof}
\subsection{Context-free grammar for $B_n(x;q)$}
\hspace*{\parindent}

For a permutation $\pi\in B_n$, we define an {\it ascent} to be a position $i\in \{0,1,2\ldots,n-1\}$ such that
$\pi(i)<\pi(i+1)$. Let $\asc_B(\pi)$ be the number of ascents of $\pi\in B_n$. As usual, denote by $\overline{i}$
the negative element $-i$. We shall
show that the following grammar
\begin{equation}\label{grammar02}
A=\{x,y,z,u\},~G=\{x\rightarrow qxyu,y\rightarrow xyz,z\rightarrow yzu,u\rightarrow qxzu\}
\end{equation}
can be used to generate permutations of $B_n$.
Now we give a labeling of $\pi\in B_n$ as follows.
\begin{itemize}
  \item [\rm ($L_1$)]If $i$ is an ascent and $\pi(i+1)>0$, then put a superscript label $z$ and a subscript $x$ right after $\pi(i)$;
 \item [\rm ($L_2$)]If $i$ is a descent and $\pi(i+1)>0$, then put a superscript label $y$ and a subscript $u$ right after $\pi(i)$;
\item [\rm ($L_3$)]If $i$ is an ascent and $\pi(i+1)<0$, then put a superscript label $z$ and a subscript $qx$ right after $\pi(i)$;
\item [\rm ($L_4$)]If $i$ is a descent and $\pi(i+1)<0$, then put a superscript label $y$ and a subscript $qu$ right after $\pi(i)$;
\item [\rm ($L_5$)] Put a superscript label $y$ and a subscript $x$ at the end of $\pi$.
\end{itemize}
The weight of $\pi$ is defined by $$w(\pi)=xy(xz)^{\asc_B(\pi)}(yu)^{\des_B(\pi)}q^{N(\pi)}.$$

It is clear that the sum of weights of $01$ and $0\overline{1}$ is given by $D(xy)$, since
$B_1=\{0^z_x1^y_x,0^y_{qu}\overline{1}^y_x\}$ and $D(xy)=x^2yz+qxy^2u$.
To illustrate the relation between the action of the formal derivative $D$ of the grammar~\eqref{grammar02} and the insertion of $n+1$ or $\overline{n+1}$ into a permutation $\pi\in B_n$, we give the following example.
Let $\pi=02\overline{1}4\overline{6}5\overline{7}~\overline{8}\in B_8$.
The labeling of $\pi$ is given by
$$0^z_x2^y_{qu}\overline{1}^z_x4^y_{qu}\overline{6}^z_x5^y_{qu}\overline{7}~^y_{qu}\overline{8}^y_x.$$
If we insert 9 after 0, the resulting permutation is given below,
$$0^z_x9^y_u2^y_{qu}\overline{1}^z_x4^y_{qu}\overline{6}^z_x5^y_{qu}\overline{7}~^y_{qu}\overline{8}^y_x.$$
Thus the insertion of 9 after 0 corresponds to applying the rule $z\rightarrow yzu$ to the label $z$ associated with 0.
If we insert $\overline{9}$ after 0, the resulting permutation is given below,
$$0^y_{qu}\overline{9}^z_x2^y_{qu}\overline{1}^z_x4^y_{qu}\overline{6}^z_x5^y_{qu}\overline{7}~^y_{qu}\overline{8}^y_x.$$
The insertion of $\overline{9}$ after 0 corresponds to applying the rule $x\rightarrow qxyu$ to the label $x$ associated with 0.
If we insert 9 after 2, the resulting permutation is given below,
$$0^z_x2^z_x9^y_{qu}\overline{1}^z_x4^y_{qu}\overline{6}^z_x5^y_{qu}\overline{7}~^y_{qu}\overline{8}^y_x.$$
The insertion of $9$ after 2 corresponds to applying the rule $y\rightarrow xyz$ to the label $x$ associated with 2.
If we insert $\overline{9}$ after 2, the resulting permutation is given below,
$$0^z_x2^y_{qu}\overline{9}^z_{qx}\overline{1}^z_x4^y_{qu}\overline{6}^z_x5^y_{qu}\overline{7}~^y_{qu}\overline{8}^y_x.$$
The insertion of $\overline{9}$ after 2 corresponds to applying the rule $u\rightarrow qxzu$ to the label $u$ associated with 2.
In general, the insertion of $n+1$ (resp. $\overline{n+1}$) into $\pi$ corresponds
to the action of the formal derivative $D$ on a superscript label (resp. subscript label).
By induction, we get the following result.
\begin{theorem}
If $D$ is the formal derivative with respect to the grammar~\eqref{grammar02},
then
\begin{equation}\label{Dnxy}
D^n(xy)=xy\sum_{\pi\in B_n}(xz)^{\asc_B(\pi)}(yu)^{\des_B(\pi)}q^{N(\pi)}.
\end{equation}
\end{theorem}
Setting $x=y=z=1$ in~\eqref{Dnxy}, we get $D^n(xy)|_{x=y=z=1}=B_n(u;q)$.
\section{Alternating run polynomials}
\hspace*{\parindent}

For $\pi\in\msn$, a {\it left peak index} is an index $i\in[n-1]$ such that $\pi(i-1)<\pi(i)>\pi(i+1)$, where we take $\pi(0)=0$, and the entry $\pi(i)$ is called a {\it left peak}.
Let $P(n,k)$ be the number of permutations of $\msn$ with $k$ left peaks.
It is well known that the numbers $P(n,k)$ satisfy the recurrence relation
$$P({n,k})=(2k+1)P({n-1,k})+(n-2k+1)P({n-1,k-1}),$$
with the initial conditions $P(1,0)=1$ and $P(1,k)=0$ for $k\geq 1$ (see~\cite{Ma121} for instance).
We say that $\pi\in\msn$ changes
direction at position $i$ if either $\pi({i-1})<\pi(i)>\pi(i+1)$, or
$\pi(i-1)>\pi(i)<\pi(i+1)$, where $i\in\{2,3,\ldots,n-1\}$. We say that $\pi$ has $k$ {\it alternating
runs} if there are $k-1$ indices $i$ such that $\pi$ changes
direction at these positions. The {\it up-down runs} of a permutation $\pi$ are the alternating runs of $\pi$ endowed with a 0
in the front.
Let $R(n,k)$ (resp.~$M(n,k)$) be the number of permutations of $\msn$ with $k$ alternating runs (resp. up-down runs).
The numbers $R(n,k)$ and $M(n,k)$ respectively satisfy the recurrence relations
\begin{equation}\label{rnk-recurrence01}
R(n,k)=kR(n-1,k)+2R(n-1,k-1)+(n-k)R(n-1,k-2),
\end{equation}
\begin{equation}\label{rnk-recurrence02}
M(n,k)=kM(n-1,k)+M(n-1,k-1)+(n-k+1)M(n-1,k-2)
\end{equation}
for $n,k\ge 1$, where $R(1,0)=M(0,0)=M(1,1)=1$ and $R(1,k)=M(n,0)=M(0,k)=0$ for $n,k\ge 1$
(see~\cite{Ma132,Sta08}).
Let $R_n(x)=\sum_{k=1}^{n-1}R(n,k)x^k$ and $M_n(x)=\sum_{k=1}^{n}M(n,k)x^k$.
There is a unified grammatical descriptions of $R_n(x)$ and $M_n(x)$.
\begin{prop}[{\cite[Theorem~6]{Ma132}}]\label{Ma132}
If $A=\{x,y,z\}$ and
$G=\{x\rightarrow xy, y\rightarrow yz,z\rightarrow y^2\}$,
then
\begin{align*}
D^{n}(x^2)&=x^2\sum_{k=0}^{n}R(n+1,k)y^kz^{n-k},\\
D^n(x)&=x\sum_{k=1}^nM(n,k)y^kz^{n-k}.
\end{align*}
In particular, $$D^{n}(x^2)|_{z=1}=x^2R_{n+1}(y),~D^n(x)|_{z=1}=xM_n(y).$$
\end{prop}

A {\it run} of a signed permutation $\pi\in B_n$ is defined as a maximal interval of consecutive elements on which the elements of $\pi$ are monotonic in the order $\cdots<\overline{2}<\overline{1}<0<1<2<\cdots$.
The {\it up signed permutations} are signed permutations with $\pi(1)> 0$.
For example, the up signed permutation $03\overline{1}24\overline{5}\in B_5$ has four runs, i.e., $03,3\overline{1},\overline{1}24$ and $4\overline{5}$.
Let $T(n,k)$ denote the number of up signed permutations in $B_n$ with $k$ alternating runs.
Zhao~\cite[Theorem 4.2.1]{Zhao11} showed that the numbers $T(n,k)$ satisfy the
the following recurrence relation
\begin{equation}\label{tnk-recurrence03}
T(n,k)=(2k-1)T(n-1,k)+3T(n-1,k-1)+(2n-2k+2)T(n-1,k-2)
\end{equation}
for $n\geqslant 2$ and $1\leqslant k\leqslant n$, where $T(1,1)=1$ and $T(1,k)=0$ for $k>1$.
The alternating run polynomials of type $B$ are defined by $T_n(x)=\sum_{k=1}^nT(n,k)x^k$.
The first few of the polynomials $T_n(x)$ are given as follows:
\begin{align*}
T_1(x)&=x,\\
  T_2(x)&=x+3x^2,\\
  T_3(x)&=x+12x^2+11x^3,\\
  T_4(x)&=x+39x^2+95x^3+57x^4.
\end{align*}

The similarity of the recurrence relations~\eqref{rnk-recurrence01},~\eqref{rnk-recurrence02} and~\eqref{tnk-recurrence03} suggests the existence of a context-free grammar for $T_n(x)$.
We now present the third main result of this paper.
\begin{theorem}
If $A=\{x,y,z\}$ and $G=\{x\mapsto xy^2,y\mapsto yz^2, z\mapsto y^4z^{-1}\}$, then for $n\geq 1$,
\begin{equation*}
\begin{split}
D^n(x^3y)&=x^3y\sum_{k=1}^{n+1}T(n+1,k)y^{2k-2}z^{2n-2k+2},\\
D^n(xy)&=xy(y^2+z^2)\sum_{k=1}^{n}T(n,k)y^{2k-2}z^{2n-2k},\\
D^n(x^2)&=2^nx^2\sum_{k=0}^{n}M(n,k)y^{2k}z^{2n-2k},\\
D^n(x^2y^2)&=2^{n-1}x^2(y^2+z^2)\sum_{k=1}^{n}R(n+1,k)y^{2k}z^{2n-2k},\\
D^n(y^2)&=2^{n}y^2\sum_{k=0}^{\lrf{n/2}}P(n,k)y^{4k}z^{2n-4k},
\end{split}
\end{equation*}
\end{theorem}
\begin{proof}
We only prove the assertion for $D^n(x^3y)$ and the others can be proved in a similar way.
Note that $D(x^3y)=x^3y(z^2+3y^2)$ and $D^2(x^3y)=x^3y(z^4+12y^2z^2+11y^4)$.
we define $\widetilde{T}(n,k)$ by
\begin{equation}\label{Dnx-def}
D^n(x^3y)=x^3y\sum_{k=0}^{n+1}\widetilde{T}(n+1,k)y^{2k-2}z^{2n-2k+2}.
\end{equation}
Then $\widetilde{T}(2,0)=0,\widetilde{T}(2,1)=T(2,1)=1$ and $\widetilde{T}(2,2)=T(2,2)=3$.
It follows from~\eqref{Dnx-def} that
\begin{align*}
  D^{n+1}(x^3y)& =D(D^n(x^3y)) \\
              & =x^3y\sum_{k=1}^{n+1}(2k-1)\widetilde{T}(n+1,k)y^{2k-2}z^{2n-2k+4}+x^3y\sum_{k=1}^{n+1}3\widetilde{T}(n+1,k)y^{2k}z^{2n-2k+2}\\
              & +x^3y\sum_{k=1}^{n+1}(2n-2k+2)\widetilde{T}(n+1,k)y^{2k+2}z^{2n-2k}.
\end{align*}
Hence
\begin{equation}\label{Dnx3y}
\widetilde{T}(n+2,k)=(2k-1)\widetilde{T}(n+1,k)+3\widetilde{T}(n+1,k-1)+(2n-2k+6)\widetilde{T}(n+1,k-2).
\end{equation}
By comparing~\eqref{Dnx3y} with~\eqref{tnk-recurrence03}, we see that
the numbers $\widetilde{T}(n,k)$ satisfy the same recurrence relation and initial conditions as $T(n,k)$, so they agree.
\end{proof}

It follows from the {\it Leibniz's formula} that
\begin{align*}
  D^{n}(x^3y)& =\sum_{k=0}^n\binom{n}{k}D^{k}(xy)D^{n-k}(x^2)\\
              & =xyD^n(x^2)+\sum_{k=1}^n\binom{n}{k}D^{k}(xy)D^{n-k}(x^2).
\end{align*}
So the following corollary is immediate.
\begin{cor}
We have
$$T_{n+1}(x)=2^nxM_n(x)+(1+x)\sum_{k=1}^n2^{n-k}\binom{n}{k}T_k(x)M_{n-k}(x).$$
\end{cor}

By using the fact that
\begin{align*}
  D^{n}(x^2y^2)& =\sum_{k=0}^n\binom{n}{k}D^{k}(xy)D^{n-k}(xy)\\
              & =2xyD^n(xy)+\sum_{k=1}^{n-1}\binom{n}{k}D^{k}(xy)D^{n-k}(xy),
\end{align*}
we immediately get a grammatical proof of the following result.
\begin{prop}[{\cite[Theorem~13]{Chow14}}]
For $n\geqslant 2$, we have
\begin{equation*}
2^{n-1}R_{n+1}(x)=2T_n(x)+\frac{1+x}{x}\sum_{k=1}^{n-1}\binom{n}{k}T_k(x)T_{n-k}(x).
\end{equation*}
\end{prop}

Let $P_n(x)=\sum_{k=0}^{\lrf{n/2}}P(n,k)x^k$. From $$D^{n}(x^2y^2)=\sum_{k=0}^n\binom{n}{k}D^{k}(x^2)D^{n-k}(y^2),$$
we obtain
\begin{equation*}\label{RnxTnxPnx}
(1+x)R_{n+1}(x)=2x\sum_{k=0}^n\binom{n}{k}M_k(x)P_{n-k}(x^2).
\end{equation*}
Recall that B\'ona~\cite[Section 1.3.2]{Bona12} obtained the following identity:
\begin{equation*}\label{TnxRnx}
M_n(x)=\frac{1}{2}(1+x)R_n(x)~\textrm{for $n\ge 2$}.
\end{equation*}
Therefore,
$$M_{n+1}(x)=x\sum_{k=0}^n\binom{n}{k}M_k(x)P_{n-k}(x^2).$$
\section{Derangement polynomials}
\hspace*{\parindent}

There is a larger literature devoted to derangement polynomials (see~\cite{Brenti90,Chen09,Chow09,Lin15,Zeng12} for instance). We now recall some basic definitions, notation
and results.
We say that a permutation $\pi$ is a {\it derangement} if $\pi(i)\neq i$ for any $i\in [n]$.
Let $\mdn_n$ be the set of derangements of $\msn$.
The {\it derangement polynomials} are defined by
$$d_n(x)=\sum_{\pi\in\mdn_n}x^{\exc(\pi)}.$$
The polynomials $d_n(x)$ satisfy the following recurrence relation
$$d_{n+1}(x)=nx\left(d_{n}(x)+d_{n-1}(x)\right)+x(1-x)d_{n}'(x)$$
for $n\geq 1$, with the initial conditions $d_0(x)=1,d_1(x)=0$ (see~\cite{Zhang95}).
Brenti~\cite[Proposition 5]{Brenti90} showed that
\begin{equation*}\label{dxz}
\sum_{n\geq 0}d_n(x)\frac{z^n}{n!}=\frac{1-x}{e^{xz}-xe^z}.
\end{equation*}

For $\pi\in\msn$, we define
$$\fix(\pi)=\#\{i\in [n]: \pi(i)=i\},$$
$$\dc(\pi)=\#\{i\in [n]: \pi(i)<i\}.$$
Clearly, $\exc(\pi)+\fix(\pi)+\dc(\pi)=n$.
In~\cite[Section~2.2]{Dumont96}, Dumont proved that
if $A=\{x,y,z,e\}$ and $G=\{x\rightarrow xy, y\rightarrow xy,z\rightarrow xy, e\rightarrow ez\}$,
then
\begin{equation}\label{Dumont}
D^n(e)=e\sum_{\pi\in\msn}x^{\exc(\pi)}y^{\dc(\pi)}z^{\fix(\pi)}.
\end{equation}
In particular, setting $y=e=1,z=0$ in~\eqref{Dumont}, we get $D^n(e)|_{y=e=1,z=0}=d_n(x)$.

In~\cite{Bre94}, Brenti introduced a definition of type $B$ weak excedance. Let $\pi\in B_n$.
We say that $i\in [n]$ is a {\it type $B$ weak excedance} of $\pi$ if $\pi(i)=i$ or $\pi(|\pi(i)|)>\pi(i)$.
Let $\we(\pi)$ be the number of weak excedances of $\pi$.
It follows from~\cite[Theorem~3.15]{Bre94} that $$B_n(x)=\sum_{\pi\in B_n}x^{\we(\pi)}.$$

A {\it fixed point} of $\pi\in B_n$ is an index $i\in [n]$ such that $\pi(i)=i$. A {\it derangement of type $B$} is a signed permutation $\pi\in B_n$ with no fixed points.
Let $\mdn_n^B$ be the set of derangements of $B_n$. Following~\cite{Chow09}, the {\it type $B$ derangement polynomials} $d_n^B(x)$ are defined by
$$d_0^B(x)=1,~d_n^B(x)=\sum_{\pi\in \mdn_n^B}x^{\we(\pi)}\quad\textrm{for $n\ge 1$}.$$
The first few of the polynomials $d_n^B(x)$ are given as follows:
\begin{align*}
d_1^B(x)&=1,\\
d_2^B(x)&=1+4x,\\
d_3^B(x)&=1+20x+8x^2,\\
d_4^B(x)&=1+72x+144x^2+16x^3.
\end{align*}
Chow~\cite[Theorem~3.2]{Chow09} proved that
$$\sum_{n\geq 0}d_n^B(x)\frac{z^n}{n!}=\frac{1-x}{e^{(2x-1)z}-xe^z}.$$
Moreover, the polynomials $d_n^B(x)$ satisfy the following recurrence relation
\begin{equation*}\label{dnbx-recu}
d_{n+1}^B(x)=2nx\left(d_n^B(x)+d_{n-1}^B(x)\right)+d_n^B(x)+2x(1-x)\frac{d}{dx}d_n^B(x).
\end{equation*}
for $n\geq 2$, with the initial conditions $d_0(x)=1,d_1(x)=0$ (see~\cite[Proposition~3.1]{Chow09}).

Given $\pi\in \mdn_n^B$. Clearly, $$\we(\pi)=\#\{i\in [n]: \pi(|\pi(i)|)>\pi(i)\}.$$
We say that $i$ is an {\it anti-excedance} of $\pi$ if $\pi(|\pi(i)|)<\pi(i)$.
Let $\aexc(\pi)$ be the number of anti-excedances of $\pi$.
We say that $i$ is a singleton if $(\overline{i})$ is a cycle of $\pi$.
Let $\single(\pi)$ be the number of singletons of $\pi$.
Then
\begin{equation}\label{single-exc}
\we(\pi)+\aexc(\pi)+\single(\pi)=n.
\end{equation}

In the following discussion, we always write $\pi$
by using its \emph{standard cycle decomposition}, in which each
cycle is written with its largest entry last and the cycles are written in ascending order of their
last entry. For example, $\overline{3}51\overline{7}2\overline{6}~\overline{4}\in \mdn_7^B$ can be written as
$(\overline{6})(\overline{7},\overline{4})(\overline{3},1)(2,5)$.
Let $(c_1,c_2,\ldots,c_i)$ be a cycle in
standard cycle decomposition of $\pi$. We say that $c_j$ is an {\it
ascent} in the cycle if $c_j<c_{j+1}$, where $1\leq j<i$.
We say that $c_j$ is a {\it descent} in the cycle if $c_j>c_{j+1}$, where $1\leq j\leq i$ and $c_{i+1}=c_1$.
As pointed out by Chow~\cite[p.~819]{Chow09}, if $\pi\in\mdn_n^B$ with no singletons,
then $\we(\pi)$ equals the sum of the number of ascents in each cycle and $\aexc(\pi)$
equals the sum of the number of descents in each cycle.
Let $d(n,i,j)$ be the number of derangements of type $B$ with $i$ weak excedances and $j$ anti-excedances.

We can now conclude the fourth main result of this paper from the discussion above.
\begin{theorem}\label{thm-derangement}
Let $A=\{x,y,z,e\}$ and
\begin{equation}\label{grammar-derangement}
G=\{x\rightarrow xy^2,y\rightarrow x^2y, z\rightarrow x^2y^2z^{-3}, e\rightarrow ez^4\}.
\end{equation}
Then
\begin{equation*}\label{Dne-derangement}
D^n(e)=e\sum_{\pi\in \mdn_n^B}x^{2\we(\pi)}y^{2\aexc(\pi)}z^{4\single(\pi)}.
\end{equation*}
Equivalently,
\begin{equation}\label{Dne-derangement}
D^n(e)=e\sum_{i,j\geq0}d(n,i,j)x^{2i}y^{2j}z^{4(n-i-j)}.
\end{equation}
Setting $y=z=1$ in~\eqref{Dne-derangement}, we get $$D^n(e)|_{y=z=1}=ed_n^B(x^2).$$
\end{theorem}
\begin{proof}
Now we give a labeling of $\pi\in \mdn_n^B$ as follows.
\begin{itemize}
  \item [\rm ($L_1$)]If $i$ is a singleton, then put a superscript label $z^4$ right after $\overline{i}$;
 \item [\rm ($L_2$)]If $c_i$ is an ascent in a cycle, then put a superscript label $x^2$ right after $c_i$;
\item [\rm ($L_3$)]If $c_i$ is a descent in a cycle, then put a superscript label $y^2$ right after $c_i$;
\item [\rm ($L_4$)]Put a superscript label $e$ in the front of $\pi$.
\end{itemize}
The weight of $\pi$ is defined by $$w(\pi)=ex^{2\we(\pi)}y^{2\aexc(\pi)}z^{4\single(\pi)}.$$
Note that $\mdn_1^B=\{^e(\overline{1}^{z^4})\}$ and
$\mdn_2^B=\{^e(\overline{2}^{z^4})(\overline{1}^{z^4}),^e(\overline{1}^{x^2}2^{y^2}),^e(1^{x^2}2^{y^2}),
^e(\overline{2}^{x^2}\overline{1}^{y^2}),^e(\overline{2}^{x^2}1^{y^2})\}$.
Thus the weight of $^e(\overline{1}^{z^4})$ is given by $D(e)$ and the sum of weights of the permutations in $\mdn_2^B$ is
given by $D^2(e)$, since
$D(e)=ez^4$ and $D^2(e)=e(z^8+4x^2y^2)$.

To illustrate the relation between the action of the formal derivative $D$ of
the grammar~\eqref{grammar-derangement} and the insertion of $n+1$ or $\overline{n+1}$
into a permutation $\pi\in \mdn_n^B$, we give the following example.
Let $\pi=(\overline{6})(\overline{7},\overline{4})(\overline{3},1)(2,5)\in\mdn_7^B$.
Then $\pi$ can be labeled as
$$^e(\overline{6}^{z^4})(\overline{7}^{x^2}\overline{4}^{y^2})(\overline{3}^{x^2}1^{y^2})(2^{x^2}5^{y^2}).$$
We distinguish the following four cases:
\begin{itemize}
  \item [\rm ($c_1$)]If we insert $\overline{8}$ as a new cycle, then the resulting permutation is given below,
$$^e(\overline{8}^{z^4})(\overline{6}^{z^4})(\overline{7}^{x^2}\overline{4}^{y^2})(\overline{3}^{x^2}1^{y^2})(2^{x^2}5^{y^2}).$$
This case corresponds to applying the rule $e\rightarrow ez^4$ to the label $e$.
 \item [\rm ($c_2$)]If we insert $\overline{8}$ or $8$ into the cycle $(\overline{6})$, then the resulting permutations are given below,
$$^e(\overline{8}^{x^2}~\overline{6}^{y^2})(\overline{7}^{x^2}\overline{4}^{y^2})(\overline{3}^{x^2}1^{y^2})(2^{x^2}5^{y^2}),$$
$$^e(\overline{6}^{x^2}~8^{y^2})(\overline{7}^{x^2}\overline{4}^{y^2})(\overline{3}^{x^2}1^{y^2})(2^{x^2}5^{y^2}),$$
$$^e({6}^{x^2}~{8}^{y^2})(\overline{7}^{x^2}\overline{4}^{y^2})(\overline{3}^{x^2}1^{y^2})(2^{x^2}5^{y^2}),$$
$$^e(\overline{8}^{x^2}~{6}^{y^2})(\overline{7}^{x^2}\overline{4}^{y^2})(\overline{3}^{x^2}1^{y^2})(2^{x^2}5^{y^2}).$$
It should be noted that in the latter two permutations, we need to replace $\overline{6}$ by $6$.
This case corresponds to applying the rule $z\rightarrow x^2y^2z^{-3}$ to the label $z^4$, since $D(z^4)=4x^2y^2$.
\item [\rm ($c_3$)]If we insert $\overline{8}$ or $8$ right after $\overline{7}$, then the resulting permutations are given below,
$$^e(\overline{6}^{z^4})(\overline{7}^{y^2}~\overline{8}^{x^2}~\overline{4}^{y^2})(\overline{3}^{x^2}1^{y^2})(2^{x^2}5^{y^2}),$$
$$^e(\overline{6}^{z^4})(\overline{7}^{x^2}~{8}^{y^2}~\overline{4}^{y^2})(\overline{3}^{x^2}1^{y^2})(2^{x^2}5^{y^2}).$$
This case corresponds to applying the rule $x\rightarrow xy^2$ to the label $x^2$, since $D(x^2)=2x^2y^2$.
\item [\rm ($c_4$)]If we insert $\overline{8}$ or $8$ right after $\overline{4}$, then the resulting permutations are given below,
$$^e(\overline{6}^{z^4})(\overline{7}^{x^2}~\overline{4}^{y^2}~\overline{8}^{x^2})(\overline{3}^{x^2}1^{y^2})(2^{x^2}5^{y^2}),$$
$$^e(\overline{6}^{z^4})(\overline{7}^{x^2}~\overline{4}^{x^2}8^{y^2})(\overline{3}^{x^2}1^{y^2})(2^{x^2}5^{y^2}).$$
This case corresponds to applying the rule $y\rightarrow x^2y$ to the label $y^2$, since $D(y^2)=2x^2y^2$.
\end{itemize}
Hence
the insertion of $n+1$ or $\overline{n+1}$ into $\pi\in\mdn_n^B$ corresponds
to the action of the formal derivative $D$ on a superscript label.
It can be easily checked
that grammar~\eqref{grammar-derangement} generates all of the type $B$ derangements.
\end{proof}

Let $D$ be the formal derivative with respect to the grammar~\eqref{grammar-derangement}.
Clearly, $D(z^4)=4x^2y^2$.
It is easy to verify that
$$D^n(x^2y^2)=2^n\sum_{k=0}^n\Eulerian{n+1}{k}x^{2k+2}y^{2n-2k+2}.$$
Hence
\begin{equation}\label{Dnz4}
D^n(z^4)=4D^{n-1}(x^2y^2)=2^{n+1}\sum_{k=0}^{n-1}\Eulerian{n}{k}x^{2k+2}y^{2n-2k}~\textrm{for $n\ge 1$}.
\end{equation}
Note that
\begin{align*}
D^{n+1}(e)&=D^n(ez^4)\\
&=\sum_{k=0}^n\binom{n}{k}D^k(e)D^{n-k}(z^4)\\
&=z^4D^n(e)+\sum_{k=0}^{n-1}\binom{n}{k}D^k(e)D^{n-k}(z^4).
\end{align*}

Combining~\eqref{Dnz4} and Theorem~\ref{thm-derangement}, we get the following corollary.
\begin{cor}
For $n\geq 1$,
we have
$$d_{n+1}^B(x)=d_n^B(x)+x\sum_{k=0}^{n-1}2^{n-k+1}\binom{n}{k}d_k^B(x)A_{n-k}(x).$$
Set $d_n^B=d_n^B(1)$. We have $$d_{n+1}^B=d_n^B+n!\sum_{k=0}^{n-1}2^{n-k+1}\frac{d_k^B}{k!}.$$
\end{cor}

Moreover, from~\eqref{Dne-derangement}, the following result is immediate.
\begin{cor}
For $n\geq 0$,
we have
\begin{equation}\label{dnij-recu}
d(n+1,i,j)=d(n,i,j)+2id(n,i,j-1)+2jd(n,i-1,j)+4(n-i-j+2)d(n,i-1,j-1),
\end{equation}
with initial conditions $d(1,0,0)=1$ and $d(1,i,j)=0$ for $i\neq 0$ or $j\neq 0$.
\end{cor}

Let $$G(x,y;t)=\sum_{n\geq 0}\sum_{\pi\in\mdn_n^B}x^{\we(\pi)}y^{\aexc(\pi)}\frac{t^n}{n!}.$$
\begin{prop}\label{thm-EGF-derangement}
We have
\begin{equation}
G(x,y;t)=\frac{e^{(1-2x)t}}{1-\frac{x}{y-x}(e^{2(y-x)t}-1)}.
\end{equation}
\end{prop}
\begin{proof}
Let $d_n(x,y)=\sum_{i,j\geq 0}d(n,i,j)x^iy^i$. It follows from~\eqref{dnij-recu} that the polynomials $d_n(x,y)$ satisfy the recurrence relation
\begin{equation}\label{dnxy}
d_{n+1}(x,y)=(1+4nxy)d_n(x,y)+(2xy-4x^2y)\frac{\partial}{\partial x}d_n(x,y)+(2xy-4xy^2)\frac{\partial}{\partial y}d_n(x,y),
\end{equation}
with the initial conditions $d_0(x,y)=d_1(x,y)=1$.
By rewriting~\eqref{dnxy} in terms of the generating function $G$, we get
\begin{equation}\label{Gex-diff}
(1-4xyt)G_t=G+(2xy-4x^2y)G_x+(2xy-4xy^2)G_y.
\end{equation}
It is routine to check that $$\widetilde{G}(x,y;t)=\frac{e^{(1-2x)t}}{1-\frac{x}{y-x}(e^{2(y-x)t}-1)}$$ satisfies~\eqref{Gex-diff}.
Also, $\widetilde{G}(x,y;0)=1$ and $\widetilde{G}(0,y;t)=\widetilde{G}(x,0;t)=e^t$. Hence $G=\widetilde{G}$.
\end{proof}

Let $P(\pi)=\#\{i\in [n]: \pi(i)>0\}$. Then $N(\pi)+P(\pi)=n$.
It is not hard to check that
\begin{equation*}
\sum_{n\geq 0}\sum_{\pi\in\mdn_n^B}x^{\we(\pi)}y^{\aexc(\pi)}z^{\single(\pi)}u^{P(\pi)}v^{N(\pi)}\frac{t^n}{n!}
=\frac{e^{\left(vz-(u+v)x\right)t}}{1-\frac{x}{y-x}\left(e^{(u+v)(y-x)t}-1\right)}.
\end{equation*}

In the proof of Theorem~\ref{thm-derangement}, if we further put a superscript
label $q$ before each cycle of $\pi\in\mdn_n^B$, then we get the following result.
\begin{prop}
Let $A=\{x,y,z,e\}$ and
$G=\{x\rightarrow xy^2,y\rightarrow x^2y, z\rightarrow x^2y^2z^{-3}, e\rightarrow qez^4\}$.
Then
\begin{equation*}
D^n(e)=e\sum_{\pi\in \mdn_n^B}x^{2\we(\pi)}y^{2\aexc(\pi)}z^{4\single(\pi)}q^{\cyc(\pi)}.
\end{equation*}
\end{prop}

\end{document}